\documentclass{amsart}

\usepackage{txfonts}
\newtheorem{theorem}{Theorem}[section]
\newtheorem{proposition}[theorem]{Proposition}
\newtheorem{lemma}[theorem]{Lemma}

\newtheorem{corollary}[theorem]{Corollary}

\theoremstyle{definition}

\theoremstyle{remark}
\newtheorem{remark}[theorem]{Remark}

\numberwithin{equation}{section}

\catcode`\ç=13
\defç{\c{c}}
\catcode`\é=13
\defé{\'e}
\catcode`\à=13
\defà{\`a}
\catcode`\è=13
\defè{\`e}
\catcode`\â=13
\defâ{\^a}
\catcode`\ù=13
\defù{\`u}
\catcode`\ê=13
\defê{\^e}
\catcode`\î=13
\defî{\^\i}
\catcode`\ô=13
\defô{\^o}

\def\pd{{\rm pd}}
\def\fd{{\rm fd}}
\def\id{{\rm id}}

\begin{document}

\title[]{When an amalgamated duplication of a ring along an ideal is quasi-Frobenius }

\author[N. Mahdou]{Najib Mahdou}
\address{Department of Mathematics, Faculty of Science and Technology of Fez, Box 2202, University S. M.
Ben Abdellah Fez, Morocco}
 \email{mahdou@hotmail.com}

\author[M. Tamekkante]{Mohammed Tamekkante}
\address{Department of Mathematics, Faculty of Science and Technology of Fez, Box 2202, University S. M.
Ben Abdellah Fez, Morocco}
 \email{tamekkante@yahoo.fr}





\begin{abstract} In this paper, we
characterize an amalgamated duplication of a ring $R$ along a proper ideal $I$, $R\bowtie I$, which is   quasi-Frobenius.
\end{abstract}

\maketitle

\section{Introduction} Throughout this paper, all rings are commutative with identity
element, and all modules are unital. If $M$ is an $R$-module, we use
$\pd_R(M)$, $\id_R(M)$ and $\fd_R(M)$ to denote, respectively, the
classical projective, injective and flat dimensions of $M$. It is
convenient to use \textquotedblleft local\textquotedblright to refer
to (not necessarily Noetherian) rings with a unique maximal ideal.
 With an idempotent element we means element $a\in R$ such that $a^2=a$.\\
\emph{The amalgamated duplication of a ring $R$ along  an
$R$-submodule of the total ring of quotients $T(R)$}, introduced by
D'Anna and Fontana and denoted by $R\bowtie E$ (see \cite{D'Anna
0,D'Anna,D'Anna basic}), is  the following subring of $R\times T(R)$
(endowed with the usual componentwise operations):
$$R\bowtie E:=\{(r,r+e)\mid  r\in R\ \text{ and } e\in E\}.$$

It is obvious that, if in the $R$-module $R\oplus E$ we introduce a
multiplicative structure by setting $(r,e)(s, f):=(rs, rf+se+ef)$,
where $r,s\in R$ and $e,f\in E$, then we get the ring isomorphism
$R\bowtie E\cong R\oplus E$. When $E^2=0$, this new construction
coincides with the \emph{Nagata's idealization}. One main difference
between  this  constructions, with respect to the idealization (or
with respect to any commutative extension, in the sense of Fossum)
is that the ring $R\bowtie E$ can be a reduced ring and it is always
reduced if R is a domain (see \cite{D'Anna 0,D'Anna basic}). If
$E=I$ is an ideal in $R$, then the ring $R\bowtie I$ is a subring of
$R\times R$. This extension has been studied, in the general case,
and from the different point of view of pullbacks, by D'Anna and
Fontana \cite{D'Anna basic}. As it happens for the idealization, one
interesting application of this construction is the fact that it
allows to produce rings satisfying (or not satisfying) preassigned
conditions. Recently, D'Anna proved that, if $R$ is a local
Cohen-Macaulay ring with canonical module $\omega_R$, then $R\bowtie
I$ is a Gorenstein ring if and only if $I\cong \omega_R$ (see
\cite{D'Anna 0}). Note also that this construction has already been
applied, by Maimani and Yassemi for studying questions concerning
the diameter and girth of the zero-divisor graph of a ring (see
\cite{Yassemi}).\\
Recently in \cite{chhitimahdou}, the authors study some homological
properties and
coherence of the \emph{amalgamated duplication} of a ring along an ideal.\\
The main result of this paper is stated as follows:
\begin{theorem}\label{main result} Let $R$ be a ring and $I$ a proper ideal of $R$. Then, $R\bowtie I$ is quasi-Frobenius if, and only if, $R$ is quasi-Frobenius and $I=Ra$ where $a$ is an idempotent element of $R$.
\end{theorem}

\section{proof of Main result}
To proof this Theorem  we need some Lemmas:
\begin{lemma}\label{lemma1} A direct product $\displaystyle\prod_{i=1}^nR_i$ of rings is  quasi-Frobenius if, and only if, each factor $R_i$ is quasi-Frobenius.
\end{lemma}
\begin{proof}Follows directly from  \cite[Proposition 2.6]{Bennis and Mahdou2} and \cite[Theorem 3.1]{Bennis and Mahdou3}.
\end{proof}
\begin{lemma}\label{implication 1}
Let $R$ be a ring and $I$ be a nonzero ideal of $R$. If $R\bowtie
I$ is quasi-Frobenius, then so is $R$.
\end{lemma}
\begin{proof} Suppose that $R\bowtie
I$ is quasi-Frobenius. If $I=R$, we have  $R\bowtie R=R\times R$.
Thus, from Lemma \ref{lemma1}, $R$ is quasi-Frobenius. So, we may
assume that $I$ is a proper ideal. From \cite[Theorem
1.50]{Nicolson}, $R\bowtie I$ is Noetherian and $Ann_{R\bowtie
I}(Ann_{R\bowtie I}(\mathcal{J}))=\mathcal{J}$ for every ideal
$\mathcal{J}$ of $R\bowtie I$ where $Ann_{R\bowtie I}(-)$ means
the annihilator over $R\bowtie I$. Then, from \cite[Proposition
2.1(4)]{D'Anna}, $R$ is
 also Noetherian. Now, let $J$ be an ideal of $R$. We claim that
$Ann_R(Ann_R(J))=J$. The inclusion $J\subseteq Ann_R(Ann_R(J))$ is
clear, so we have to prove the converse inclusion. The set $J\bowtie
I:=\{(j,j+i)\mid  j\in J \text{ and } i\in I\}$ is an ideal of $R\bowtie
I$. Thus,   $Ann_{R\bowtie I}(Ann_{R\bowtie
I}(J\bowtie I))=J\bowtie I$. Let $(r,r+i)\in Ann_{R\bowtie
I}(J\bowtie I)$. Then, for all $j\in J$ we have
$(rj,rj+ij)=(r,r+i)(j,j)=(0,0)$ since $(j,j)\in J\bowtie I$.
Therefore, $r,i\in Ann_R(J)$. So, if  $x\in Ann_R(Ann_R(J))$ we have
$(x,x)(r,r+i)=(xr,xr+xi)=(0,0)$. Then, $(x,x)\in Ann_{R\bowtie
I}(Ann_{R\bowtie I}(J\bowtie I))=J\bowtie I$. So, $x\in J$ as
desired. Consequently, from \cite[Theorem 1.50]{Nicolson}, $R$ is a
quasi-Frobenius ring.
\end{proof}
\begin{lemma}\label{qf local regular elemnt}
Let $R$ be a local ring and $I$ be a nonzero ideal of $R$. If
$R\bowtie I$ is quasi-Frobenius then, $R$ is quasi-Frobenius and
$I=R$.
\end{lemma}
\begin{proof} Suppose that $R\bowtie I$ is a quasi-Frobenius ring.
Then, by Lemma \ref{implication 1}, $R$ is also  quasi-Frobenius.
Moreover,  from the isomorphism of $R$-modules
$I\cong_RHom_{R\bowtie I}(R,R\bowtie I)$ (\cite[Proposition
3]{D'Anna 0}), we deduce that  $I$ is an injective $R$-module
since $R\bowtie I$ is self injective (From \cite[Theorem
1.50]{Nicolson}). Thus, $I$ is projective since $R$ is
quasi-Frobenius and from \cite[Theorem 7.56]{Nicolson}. Hence,
there exists a nonzero divisor element $x\in R$ such that $I=xR$
 since $R$ is local. We have  the following descendent chain of
ideals:
$$....\subseteq x^3R\subseteq x^2R \subseteq xR$$
Since $R$ is Artinian (from \cite[Theorems 1.50]{Nicolson}), this
chain is finite an so there is an integer $n$ such that
$x^{n+1}R=x^nR$. Then, there is an element $a\in R-\{0\}$ such
that $x^n=x^{n+1}a$ and so $x^n(1-xa)=0$. Thus, $1=xa$ since $x$
is a nonzero divisor element of $R$. Consequently, $x$ is a unit
element and $I=R$.
\end{proof}
 Using the above Lemma we have the following direct Corollary:
\begin{corollary} If $R$ is local ring and $I$ a nonzero proper ideal of $R$. Then, $R\bowtie I$ is never quasi-Frobenius.
\end{corollary}
\begin{lemma}\label{division}
Let $\{R_{i}\}_{1\leq i\leq n}$ be a family of rings. For each $1\leq
i\leq n$, let $I_i$ be   an ideal of $R_i$. Then, there is a naturel
isomorphism of rings  $$\prod_{i=1}^n(R_i\bowtie I_i)\cong
(\prod_{i=1}^nR_i)\bowtie (\prod_{i=1}^nI_i)$$
\end{lemma}

\begin{proof} The proof is done by induction on $n$ and it
suffices to check it for $n=2$. \\
But, it is clear that the map: \\
$$\begin{array}{cccc}
  \zeta: & (R_1\bowtie I_1)\times (R_2\bowtie
I_2) & \rightarrow & (R_1\times R_2)\bowtie (I_1\times I_2) \\
   & ((r_1,r_1+i_1),(r_2,r_2+i_2)) & \mapsto & ((r_1,r_2),(r_1,r_2)+(i_1,i_2))
\end{array}$$
is an isomorphism of rings, as desired.
\end{proof}
\begin{lemma}\label{astuce}
Let $R$ be a ring and $I$, $J$ be a two ideals of $R$ such that
$I\cap J =\{0\}$ and $I+J=R$. Then, we have the following
isomorphism of rings $R\bowtie I\cong R/I\times (R/J)^2$ and
$R\cong R/I\times R/J$.
\end{lemma}
\begin{proof}
The proof will be elementary. Since $I+J=R$,  let $i_0\in I$ and
$j_0\in J$ such that $1=i_0+j_0$. To see the first isomorphism we
consider the following map
$$\begin{array}{cccc}
  \psi: & R\bowtie I & \rightarrow & R/I\times (R/J)^2 \\
  & (r,r+i) & \mapsto & (\overline{r},(\widehat{r},\widehat{r+i}))
\end{array}$$
It is easy to see that $\psi$ is an homomorphism of rings.
Moreover, if $\psi(r,r+i)=0$, we will have $r,i\in I\cap J
=\{0\}$. On the other hand, if
$(\overline{r},(\widehat{r'},\widehat{r"}))$ is an element of $R/I
\times (R/J)^2$, we have that
$(\overline{r},(\widehat{r'},\widehat{r"}))=\psi((rj_0+r'i_0,rj_0+r'i_0+r"i_0))$.
Hence, $\psi$ is an isomorphism of rings.\\
Also, it is easy to check that
$$\begin{array}{cccc}
  \phi: & R & \rightarrow & R/I\times R/J \\
  & r & \mapsto & (\overline{r},\widehat{r})
\end{array}$$
is an homomorphism of rings which is injective (since
$(\overline{r},\widehat{r})=0$ implies that $r\in  I\cap J
=\{0\}$) and surjectif (since for every $r,r'\in R$, we have
$(\overline{r},\widehat{r'})=\phi(j_0r+i_0r')$). Hence, $\phi$ is
an isomorphism of rings.
\end{proof}
\begin{remark}
 A particular case of Lemma \ref{astuce}, is by taken $I=eR$ where $e$ is an idempotent element and $J=(1-e)R$. Clearly, $I+J=R$ and $I\cap J=0$.
 \end{remark}
\begin{proof}[Proof of Theorem \ref{main result}] Assume that
$R\bowtie I$ is quasi-Frobenius. Then, by Lemma  \ref{implication 1}, $R$ is quasi-Frobenius. Thus, from \cite[Proposition 1.7]{Ouarghi},
$R=\displaystyle\prod_{i=1}^n R_i$ where $R_i$ is local quasi-Frobenius. Hence, $I$
has the form  $I=\displaystyle
\prod_{i=1}^nI_i$ where $I_i$
is an ideal of $R_i$ for each $1\leq i\leq n$. Thus, by Lemma \ref{division}, $R\bowtie
I\cong \displaystyle\prod_{i=1}^nR_i\bowtie I_i$. Hence, by Lemma \ref{lemma1}, $R_i\bowtie I_i$ is quasi-Frobenius for each $i$. Then, by
Lemma \ref{qf local regular elemnt},  $I_i=R_i$ or $I_i=0$. Hence, $I=(a_1,...,a_n)R$ where $a_i\in \{0,1\}$.
 But since $I$ is a proper ideal $(a_1,...,a_n)\neq 0$. Clearly, $(a_1,...,a_n)$ is an idempotent element of $R$, as desired.\\
Conversely, suppose that $R$ is a quasi-Frobenius ring and let $a$ be an  idempotent element of $R$. Then, by lemma \ref{astuce}, $R\cong R/(a)\times R/(1-a)$. Thus, $R/(a)$ and $R/(1-a)$ are quasi-Frobenius and so is $R\bowtie I\cong R/(a)\times (R/(1-a))^2$, as desired.
\end{proof}

\begin{proposition}\label{idemp}
Let $R$ be a ring an $I$ a proper ideal of $R$. Then, $R$ is
projective as an $R\bowtie I$-module if, and only if, $I=Re$ where
$e$ is an idempotent element of $R$.
\end{proposition}
\begin{proof}
Consider the short exact sequence of $R\bowtie I$-modules:
$$(\star)\quad 0\longrightarrow 0\times I\stackrel{\iota}\longrightarrow R\bowtie
I\stackrel{\varepsilon}\longrightarrow R \longrightarrow 0$$ where
$i$ is the injection and $\varepsilon(r,r+i)=r$. It is obvious that
$R$ is projective (as an $R\bowtie I$-module) if, and only if,
$(\star)$ split. That imply the existence of an $R\bowtie
I$-morphism $\pi: R\rightarrow R\bowtie I$ such that
$\varepsilon\circ \pi= id(R)$. If $\pi(1):=(a,a+j)$, we get
$1=\varepsilon\circ \pi(1)=\varepsilon(a,a+j)=a$. Hence,
  for an arbitrary element $r\in R$ and all $i\in I$,
$\pi(r)=\pi((r,r+i).1)=(r,r+i)\pi(1)=(r,r+i)(1,1+j)=(r,r(1+j)+i(1+j))$.
But $\pi$ must be well defined. Thus, i(1+j)=0 and then $i=-ij$ for
any $i\in I$. In particular, $j=-j^2$ and so $i=ij^2$. Hence,
$I=Rj^2$ and $j^4=(j^2)^2=(-j)^2=j^2$ (i.e; $j^2$ is an
idempotent element of $R$).\\
Conversely, if $I$ is generated by an idempotent $e$, it is easy, by
considering the $R\bowtie (e)$-morphism $\pi: R\rightarrow R\bowtie
(e)$ defined by $\pi(1)=(1,1-e)$, to check that $(\star)$ splits.
\end{proof}
\begin{corollary}
Let $R$ be a quasi-Frobenius ring and $I$ a proper ideal of $R$. Then, $R\bowtie I$ is quasi-Frobenius if, and only if, $R$ is a projective $R\bowtie I$-module.
\end{corollary}
\begin{proof}
Follows directly from Theorem \ref{main result} and Proposition \ref{idemp}.
\end{proof}

\bibliographystyle{amsplain}

\end{document}